\documentclass[preprint,10pt,authoryear]{elsarticle}

\usepackage[margin=3cm]{geometry}
\usepackage{bbm}
\usepackage{amssymb,mathrsfs,amsmath,amsthm}
\usepackage[colorlinks=true]{hyperref}
\usepackage{dsfont}
\usepackage[latin1]{inputenc}

\newcommand\NoBlackBoxes{\global\overfullrule0pt}
\NoBlackBoxes

\newtheorem{definition}{Definition}[section]
\newtheorem{theorem}[definition]{Theorem}
\newtheorem{lemma}[definition]{Lemma}

\newtheorem{example}[definition]{Example}

\newcommand{\R}{{\mathbb R}}
\newcommand{\Z}{{\mathbb Z}}

\newcommand{\E}{{\mathbb E}}

\newcommand{\cov}{\mathrm{Cov}}

\newcommand{\X}{\textbf{X}}
\newcommand{\tr}{\mathrm{tr}}
\newcommand{\nc}{\mathcal{N}\mathcal{P}\mathcal{P}(k)}
\newcommand{\bpi}{\boldsymbol{\pi}}

\journal{Statistics \& Probability Letters}

\begin{document}

\begin{frontmatter}

\title{The semicircle law for matrices with ergodic entries}

\author{Matthias L\"owe}
\address{Fachbereich Mathematik und Informatik,
University of M\"unster,
Einsteinstra\ss e 62,
48149 M\"unster,
Germany}

\ead{maloewe@math.uni-muenster.de}



\date{\today}



\begin{abstract}
We study the empirical spectral distribution (ESD) of symmetric random matrices with ergodic entries on the diagonals. We observe that for entries with correlations that decay to 0, when the distance of the diagonal entries becomes large the limiting ESD is the well known semicircle law. If it does not decay to 0 (and have the same sign) the semicircle law cannot be the limit of the ESD. This is good agreement with results on exchangeable processes analysed in \cite{FL12} and \cite{HKW15}.
\end{abstract}

\end{frontmatter}

\section{Introduction}
Studying the (ESD) of large random matrices has been a major research topic in probability over the past decades. Already Wigner \cite{Wigner} and Arnold \cite{Arnold}
showed that the ESD of symmetric or Hermitian random matrices with i.i.d. entries, otherwise, under appropriate
scaling converges to the semi-circle law. Other versions of Wigner's semi-circle law were proved for matrix ensembles with entries
drawn according to weighted Haar measures on classical groups see e.g. \cite{AGZ}.
This note analyzes the universality of the semi-circle law for matrices with stochastically correlated entries.
Similar attempts have been made in \cite{schulzbaldes}, where a limited number of correlated entries are admitted,
\cite{goetze_tikho}, where a martingale structure of the entries is imposed, \cite{FL13}, where the diagonals of the (symmetric) matrices were filled with independent Markov chains or in \cite{LS17}, where the upper triangular part of the matrix is filled with a one-dimensional Markov chain. On the other hand, \cite{FL12} and \cite{HKW15} study matrices where either the diagonals or the entire matrix is built of exchangeable random variables. In particular, in \cite{FL12} it was shown that there is a phase transition: If the correlation of the exchangeable random variables go to 0, the limit of the ESD is the semi-circle law, while otherwise it can be described in terms of a free convolution of the semi-circle law with a limiting law obtained in \cite{brycdembo}.
In this note we will answer a question by C. Deninger (private communication) and extend the results from \cite{FL12} and \cite{FL13} to ergodic sequences of random variables. We will see that the convergence of the ESD only depends on the correlations of the entries. If they converge to 0, as the distance grows to infinity, the limiting ESD is again the semicircle law, while this is not the case if the correlations have the same sign and are bounded away from 0. This result is obtained by weakening the conditions in \cite{FL13} and is in good agreement with the results from \cite{FL12}.

The rest of the note is organized as follows. In the second section we will formalize the situation we want to consider and state our main result. Section 3 is devoted to the proof, that is based on a moment method. Finally, we will also give an example.

{\bf Acknowledgement}: This work was supported by SFB 878.

I am very grateful to remarks of an anonymous referee, who spotted a gap in an earlier version of the paper and suggested the correct solution.
\section{Main Results}

The central object in this note is the semi-circle distribution. Its density is given by
$\frac{d\mu}{dx}(x)=\frac 1 {2\pi} \sqrt {4-x^2} \mathbbm{1}_{[-2,2]}(x)$.
To fill our random matrix we will consider random variables
$$\left\{a(p,q), 1\leq p\leq q< \infty\right\} \mbox{   with values in $\R$.}$$
We assume that $\mathbb{E}\left[a(p,q)\right]=0$, $\mathbb{E}\left[a(p,q)^{2}\right] = 1$ and, to be able to employ the moment method,
\begin{equation}\label{moment}
 m_k:=\sup_{n\in\mathbb{N}} \max_{1\leq p\leq q\leq n} \mathbb{E}\left[\left|a(p,q)\right|^{k}\right] < \infty, \quad k\in\mathbb{N}.
\end{equation}
Moreover we construct the families $\left\{a(p,p+r), p\in\mathbb{N}\right\}$, $r\in\mathbb{N}_0$, the diagonals of our future matrices, as independent random variables. On the diagonals we will assume that for every pair $r$ the process $Z^{r}_p:=(a(p,p+r))_{p \in \Z}$ is a stationary stochastic process. We will assume that this process is ergodic.
For $n\in \mathbb{N}$, define the symmetric random $n\times n$ matrix $\X_n$:
\begin{equation*}
 \X_n (q,p) = \X_n (p,q) = \frac{1}{\sqrt{n}} a(p,q), \qquad 1\leq p\leq q\leq n,
\end{equation*}
We will denote the (real) eigenvalues of $ \textbf{X}_{n}$ by $\lambda_1^{(n)} \le \lambda_2^{(n)} \le \ldots \lambda_n^{(n)}$. Let $\mu_n$ be the empirical eigenvalue distribution, i.e.
$
\mu_n = \frac{1}{n} \sum_{k=1}^n \delta_{\lambda_k^{(n)}}.$
We want to show:
\begin{theorem}
Assume that $\X_n$ satisfies the above conditions and the diagonal processes $(Z^{(r)}_p)$ have correlations that decay to 0, i.e.
$\E(Z^{(r)}_1 Z^{(r)}_p) \to 0$ for all $r$ as $p \to \infty$ (recall that $\E Z^{(r)}_1 =0$).
Then the ESD of $\textbf{X}_n$ converges weakly to the semi-circle distribution, i.e.
$
\mu_n \Rightarrow \mu \mbox{ as } n \to \infty
$
both, in expectation and $\mathbb{P}-\mbox{almost surely}$.

If the correlations of $(Z^{(r)}_p)$ all have the same sign and are bounded away from 0 the ESD of $\textbf{X}_n$ does not converge to $\mu$.
\label{main}
\end{theorem}


\section{Proof of Theorem \ref{main}}
The ideas in the first section of the proof partially follow \cite{FL13} which in turn are based on considerations in \cite{schulzbaldes}. For the reader's convenience and since we will need the arguments for the second part of the proof as well, we will illustrate the main steps.
We will use the method of moments to prove Theorem \ref{main}. To this end, recall that the even moments of the semicircle law are given
by the Catalan numbers $C_{\frac{k}{2}} = \frac{k!}{\frac{k}{2}!\left(\frac{k}{2}+1\right)!}$, while the odd moments are $0$.
The weak convergence of the expected empirical distribution will follow from the convergence of the expected moments
\begin{equation*}
 \lim_{n\to\infty} \frac{1}{n} \E\left[\tr\left(\X_{n}^{k}\right)\right] = \begin{cases} 0, & \text{if} \ k \ \text{is odd}, \\ C_{\frac{k}{2}}, & \text{if} \ k \ \text{is even,} \end{cases}
\end{equation*}
where $\tr(\cdot)$ denotes the trace operator.
To see that this is true, we obviously need to compute the moments. To this end consider the set $\mathcal{T}_n(k)$ of $k$-tuples of so-called consistent pairs, i.e. elements of the form $\left(P_1,\ldots,P_k\right)$ with $P_j = (p_j,q_j) \in \left\{1,\ldots,n\right\}^2$ satisfying $q_j = p_{j+1}$ for any $j=1,\ldots,k$, where $k+1$ is identified with $1$. Then, we have
\begin{equation*}
 \frac{1}{n} \E\left[\tr\left(\X_{n}^{k}\right)\right] = \frac{1}{n^{1+\frac{k}{2}}} \sum_{\left(P_1,\ldots,P_k\right)\in\mathcal{T}_n(k)} \E\left[a(P_1)\cdot \ldots \cdot a(P_k)\right].
\end{equation*}
Define $\mathcal{P}(k)$ to be the set of all partitions $\pi$ of $\left\{1,\ldots,k\right\}$ and write $i \sim_\pi j$ for $i,j \in \left\{1,\ldots,k\right\}$, iff
$i$ and $j$ belong to the same set of the partition $\pi$.
We call $\left(P_1,\ldots,P_k\right)\in\mathcal{T}_n(k)$ a $\pi$-consistent sequence if
$\left|p_i - q_i\right| = \left|p_j - q_j\right|$ iff $i$ and $j$ are equivalent w.r.t $\pi$.
The idea of this construction is that $a(P_{i_1}),\ldots,a(P_{i_l})$ are independent when $i_1,\ldots,i_l$ belong to $l$ different blocks of $\pi$. Let $S_n(\pi)$ be the set of all $\pi$ consistent sequences. 
Hence
\begin{equation*}
 \frac{1}{n} \E\left[\tr\left(\X_{n}^{k}\right)\right] = \frac{1}{n^{1+\frac{k}{2}}} \sum_{\pi \in \mathcal{P}(k)} \sum_{\left(P_1,\ldots,P_k\right)\in S_n(\pi)} \E\left[a(P_1)\cdot \ldots \cdot a(P_k)\right].
\label{sum}
\end{equation*}
Now fix  $k\in\mathbb{N}$. For $\pi\in\mathcal{P}(k)$ let $\# \pi$ denote the number of equivalence classes of $\pi$.

Let us first consider the case $\# \pi > \frac{k}{2}$.
Then there is at least one equivalence class with a single element $l$. By independence of the elements in different equivalence classes this implies for any sequence $\left(P_1,\ldots,P_k\right)\in S_n(\pi)$ that
$ \E\left[a(P_1)\cdot \ldots \cdot a(P_k)\right] = \E\Big[\prod_{i\neq l}a(P_i)\Big] \cdot \E\left[a(P_l)\right] = 0.
$
%
Next we examine the case
$r:= \# \pi < \frac{k}{2}$.
To compute $\# S_n(\pi)$, fix $\left(P_1,\ldots,P_k\right)\in S_n(\pi)$. To this end, first choose the pair $P_1 = (p_1,q_1)$. There are at most $n$ possibilities to assign a value to $p_1$ and $n$ possibilities for $q_1$. For $P_2 = (p_2,q_2)$, note that the consistency of the pairs implies $p_2 = q_1$. If $1\sim_\pi 2$, the condition $\left|p_1 - q_1\right| = \left|p_2 - q_2\right|$ allows at most two choices for $q_2$. Otherwise, we have at most $n$ possibilities. We now proceed sequentially to determine the remaining pairs: When arriving at some index $i$, we check whether $i$ is equivalent to any preceding index $1,\ldots,i-1$. If this is true, then we have at most two choices for $P_i$, otherwise, we have $n$. Since there are exactly $r$ different equivalence classes, we see that
 $\# S_n(\pi) \leq C \cdot n^{r+1}$
with a constant $C=C(r,k)$ depending on $r$ and $k$.
H\"{o}lder's inequality implies that for all sequences $(P_1,\ldots,P_k)$,
\begin{equation}
 \left|\E\left[a(P_1)\cdot \ldots \cdot a(P_k)\right] \right| \leq \left[\E\left|a(P_1)\right|^{k}\right]^{\frac{1}{k}} \cdot \ldots \cdot \left[\E\left|a(P_k)\right|^{k}\right]^{\frac{1}{k}} \leq m_k <\infty.
\label{holder}
\end{equation}
As $r< \frac{k}{2}$, we get
\begin{equation*}
 \frac{1}{n^{1+\frac{k}{2}}} \sum_{\underset{\# \pi < \frac{k}{2}}{\pi \in \mathcal{P}(k),}} \sum_{\left(P_1,\ldots,P_k\right)\in S_n(\pi)} \left|\E\left[a(P_1)\cdot \ldots \cdot a(P_k)\right] \right| \leq C \cdot \frac{1}{n^{1+\frac{k}{2}}} \cdot n^{r+1} = o(1).
\end{equation*}
Thus
$ \lim_{n\to\infty} \frac{1}{n} \E\left[\tr\left(\X_{n}^{k}\right)\right] = \lim_{n\to\infty} \frac{1}{n^{1+\frac{k}{2}}} \sum_{\underset{\# \pi = \frac{k}{2}}{\pi \in \mathcal{P}(k),}} \sum_{\left(P_1,\ldots,P_k\right)\in S_n(\pi)} \E\left[a(P_1)\cdot \ldots \cdot a(P_k)\right],
$
if the limits exist.

For odd $k$ this automatically implies
 $\lim_{n\to\infty} \frac{1}{n} \E\left[\tr\left(\X_{n}^{k}\right)\right] = 0$.

\noindent For even $k$ denote by $\mathcal{P}\mathcal{P}(k)\subset \mathcal{P}(k)$
the set of all pair partitions of $\left\{1,\ldots,k\right\}$. Thus $\# \pi = \frac{k}{2}$ for all $\pi \in \mathcal{P}\mathcal{P}(k)$.
However, if $\#\pi = \frac{k}{2}$ but $\pi \notin \mathcal{P}\mathcal{P}(k)$, we see that $\pi$ has at least one equivalence class with a
single element and hence, as above, the expectation corresponding to the $\pi$ consistent sequences will become 0, implying that
\begin{equation*}
 \lim_{n\to\infty} \frac{1}{n} \E\left[\tr\left(\X_{n}^{k}\right)\right] = \lim_{n\to\infty} \frac{1}{n^{1+\frac{k}{2}}} \sum_{\pi\in\mathcal{P}\mathcal{P}(k)} \sum_{\left(P_1,\ldots,P_k\right)\in S_n(\pi)} \E\left[a(P_1)\cdot \ldots \cdot a(P_k)\right],
\end{equation*}
if the limits exist.
Next we want to fix a $\pi\in\mathcal{P}\mathcal{P}(k)$ and analyze the set $S_n(\pi)$. To this end
we denote by
$S_{n}^{*}(\pi) \subseteq S_n(\pi)$ the set of $\pi$ consistent sequences $(P_1,\ldots,P_k)$ satisfying
$
 i\sim_\pi j \Longrightarrow  q_i - p_i = p_j - q_j$
for all $i\neq j$. Then, we have
$ \# \left(S_n(\pi)\backslash S_{n}^{*}(\pi)\right) = o\left(n^{1+\frac{k}{2}}\right).
$
This is indeed well known. The proof can be found in \cite{FL13}, Lemma 3.1 or \cite{brycdembo}, Proposition 4.4.
\noindent We hence obtain
\begin{equation*}
 \lim_{n\to\infty} \frac{1}{n} \E\left[\tr\left(\X_{n}^{k}\right)\right] = \lim_{n\to\infty} \frac{1}{n^{1+\frac{k}{2}}} \sum_{\pi\in\mathcal{P}\mathcal{P}(k)} \sum_{\left(P_1,\ldots,P_k\right)\in S_{n}^{*}(\pi)} \E\left[a(P_1)\cdot \ldots \cdot a(P_k)\right],
\end{equation*}
if the limits exist. We now come to the heart of the proof and call a pair partition $\pi \in \mathcal{P}\mathcal{P}(k)$ crossing if there are indices $i<j<l<m$ with $i\sim_\pi l$ and $j\sim_\pi m$. Otherwise, we call $\pi$ non-crossing. The set of all non-crossing pair partitions is denoted by $\nc$.
It is decisive to see that the set of crossing partitions does not contribute in the limit.
\begin{lemma}
If the diagonal processes $(Z^{(r)}_p)$ have correlations that decay to 0 as $p \to \infty$, for all $\pi \in \mathcal{P}\mathcal{P}(k) \backslash \nc$, we have
\begin{equation*}
 \sum_{\left(P_1,\ldots,P_k\right)\in S_{n}^{*}(\pi)} \E\left[a(P_1)\cdot \ldots \cdot a(P_k)\right] = o\left(n^{\frac{k}{2}+1}\right).
\end{equation*}
\label{crossing}
\end{lemma}
\begin{proof}
Let $\pi$ be crossing and consider $\left(P_1,\ldots,P_k\right)\in S_{n}^{*}(\pi)$. If there is $l\in\left\{1,\ldots,k\right\}$ with $l\sim_\pi l+1$, we have
 $a(P_l) = a(P_{l+1})$,
since $q_l = p_{l+1}$ by consistency and then $p_l = q_{l+1}$ by definition of $S_{n}^{*}(\pi)$. Thus,
 $\E\left[a(P_l) \cdot a(P_{l+1})\right] = 1$.
After erasing $P_l, P_{l+1}$, the remaining sequence $\left(P_1,\ldots, P_{l-1},P_{l+2}, \ldots,P_k\right)$ is still consistent.
There are at most $n$ choices for $q_l = p_{l+1}$, hence we obtain
 $\# S_{n}^{*}(\pi) \leq n \cdot \# S_{n}^{*}(\pi^{(1)}),$
where $\pi^{(1)}$
is the pair partition induced by $\pi$ after eliminating the indices $l$ and $l+1$. If $r$ denotes the maximum number of index pairs that can be eliminated by removing neighbouring equivalent pairs, there are at least two pairs left, since $\pi$ is non-crossing. Thus $r\leq\frac{k}{2}-2$. Hence,
$
 \# S_{n}^{*}(\pi) \leq n^r \cdot \# S_{n}^{*}(\pi^{(r)}),$
where $\pi^{(r)}$
is the pair partition induced by $\pi$. Thus,
\begin{align}
\sum_{\left(P_1,\ldots,P_k\right)\atop \in S_{n}^{*}(\pi)} \left|\E\left[a(P_1)\cdot \ldots \cdot a(P_k)\right]\right|
\leq n^r \sum_{(P^{(r)}_1,\ldots,P^{(r)}_{k-2r})\atop \in S_{n}^{*}(\pi^{(r)})} \left|\E\left[a(P^{(r)}_1)\cdot \ldots \cdot a(P^{(r)}_k)\right]\right|.
\label{estimate}
\end{align}
 Now we choose a pair $(i,j)$ such that $i\sim_{\pi^{(r)}} i+j$ and $j$ is minimal. We count the number of sequences $(P_1^{(r)},\ldots,P_{k-2r}^{(r)})\in S_{n}^{*}(\pi^{(r)})$ with fixed $p^{(r)}_i$ and $q^{(r)}_{i+j}$. $q^{(r)}_i$ takes one of $n$ possible values. Then
$ p^{(r)}_{i+j} = q^{(r)}_i - p^{(r)}_i + q^{(r)}_{i+j}.$
Since $j$ is minimal, any element in $\left\{i+1,\ldots,i+j-1\right\}$ is equivalent to some element outside the set $\left\{i,\ldots,i+j\right\}$. There are $n$ possibilities to fix $P^{(r)}_{i+1}$ as $p^{(r)}_{i+1}=q^{(r)}_i$ is already fixed. In the same way, we have $n$ possibilities for the choice of any pair $P^{(r)}_l$ with $l\in \left\{i+2,\ldots,i+j-2\right\}$ and there is only one choice for $P^{(r)}_{i+j-1}$ since $q^{(r)}_{i+j-1}=p^{(r)}_{i+j}$ is already chosen. For all other pairs that has not yet been fixed, there are at most $n$ possibilities if it is not equivalent to one pair that has already been chosen. Otherwise, there is only one possibility. So, if $p^{(r)}_i$ and $q^{(r)}_{i+j}$ are fixed, we have at most
 $n \cdot n^{j-2} \cdot n^{\frac{k}{2}-r-j} = n^{\frac{k}{2}-r-1}$
possibilities to choose the rest of the sequence $(P^{(r)}_1,\ldots,P^{(r)}_{k-2r})\in S_{n}^{*}(\pi^{(r)})$. So \eqref{estimate} becomes 
\begin{align} \label{estimate2}
\sum_{\left(P_1,\ldots,P_k\right)\in S_{n}^{*}(\pi)} \left|\E\left[a(P_1)\cdot \ldots \cdot a(P_k)\right]\right| & \leq n^{\frac{k}{2}-1} \sum_{p^{(r)}_i,q^{(r)}_{i+j}=1}^{n} | \cov(|q^{(r)}_{i+j}-p^{(r)}_i|) |
\end{align}
where we set $\cov_{|p-q|}(t) := \cov(a(p,q),a(p+t,q+t))$ and $\cov(t):= \sup_{p,q}\cov_{|p-q|}(t)$.
This in turn can be estimated by
\begin{equation}\label{estimate3}
n^{\frac{k}{2}-1} \sum_{p^{(r)}_i,q^{(r)}_{i+j}=1}^{n} | \cov(|q^{(r)}_{i+j}-p^{(r)}_i|) | \le
n^{\frac{k}{2}} \sum_{t=0}^n \cov(t) = n^{\frac{k}{2}+1} (\frac 1 n \sum_{t=0}^n \cov(t)).
\end{equation}
Now by assumption $\cov(t) \to 0$ as $t \to \infty$, hence $(\frac 1 n \sum_{t=0}^n \cov(t))=o(1)$. Thus the right hand side of \eqref{estimate3} is $o(n^{\frac k2 +1}$ as asserted.
\end{proof}

\noindent Lemma~\ref{crossing} means that we need only to concentrate on non-crossing pair partitions.
\begin{lemma}
 Let $\pi \in \nc$. For any $\left(P_1,\ldots,P_k\right)\in S_{n}^{*}(\pi)$, we have
$ \E\left[a(P_1)\cdot\ldots\cdot a(P_k)\right] = 1.$
\label{le1}
\end{lemma}

\begin{proof}
Take $l<m$ with $m\sim_\pi l$. Since $\pi \in \nc$, $l-m-1$ is even. Hence, there is $l\leq i< j\leq m$ with $i\sim_\pi j$ and $j=i+1$. Therefore, as above, we have $a(P_i)=a(P_j)$, and $\left(P_1,\ldots, P_l,\ldots,P_{i-1},P_{i+2},\ldots,P_m,\ldots,P_k\right)$ is still consistent. Continuing like this we see that all pairs between $l$ and $m$ vanish and that the sequence $\left(P_1,\ldots,P_l,P_m,\ldots,P_k\right)$ is consistent, that is $q_l=p_m$. Then, the identity $p_l=q_m$ also holds. In particular, $a(P_l)=a(P_m)$. Since $l,m$ have been chosen arbitrarily, we obtain
$ \E\left[a(P_1)\cdot\ldots\cdot a(P_k)\right] = \prod_{\stackrel{l< m}{l\sim_\pi m}} \E\left[a(P_l)\cdot a(P_m)\right] = 1.
$
\end{proof}
We finally check:

\begin{lemma}
 For any $\pi\in\nc$, we have
\begin{equation*}
 \lim_{n\to\infty} \frac{\# S_{n}^{*}(\pi)}{n^{\frac{k}{2}+1}} = 1.
\end{equation*}
\label{le2}
\end{lemma}

\begin{proof}
The proof does not depend on the properties of our stochastic process. It can e.g. be found in \cite{FL13}.
\end{proof}
We thus arrive at
$ \lim_{n\to\infty} \frac{1}{n} \E\left[\tr\left(\X_{n}^{k}\right)\right] =
\lim_{n\to\infty} \frac{1}{n^{1+\frac{k}{2}}} \sum_{\pi\in\nc} \# S_{n}^{*}(\pi) = \# \nc.
$
Now it is well known that $\# \nc$ equals $C_{\frac{k}{2}}$, and we see that the expected ESD of $\X_n$ tends to the semi-circle law. This is the convergence in expectation. \\

The almost sure convergence follows ideas from above and also techniques from \cite{brycdembo}. The idea is to use an appropriate Markov inequality together with the Borel-Cantelli-Lemma. To this end we show that under the
conditions of Theorem~\ref{main} for any $k,n \in\mathbb{N}$,
\begin{equation}\label{Markov4}
A_4:=
 \E\left[\left(\tr\left(\X_{n}^{k}\right) - \E\left[\tr \left(\X_{n}^{k}\right)\right]\right)^4\right] \leq C \cdot n^{2}.
\end{equation}
From here the almost sure convergence follows: For any $\varepsilon>0$ and any $k,n\in\mathbb{N}$,
\begin{equation*}
 \mathbb{P}\left( \left| \frac{1}{n} \tr\X_n^k - \E \left[\frac{1}{n}\tr\X_n^k\right] \right| > \varepsilon \right) \leq \frac{C}{\varepsilon^4 n^2}.
\end{equation*}
Hence, the convergence in expectation part shown above together with the Borel-Cantelli lemma yield that
$
 \lim_{n\to\infty} \frac{1}{n} \tr\X_n^k$
almost surely equals the $k$'th moment of the semicircle distribution.
Thus, we have that, with probability $1$, the ESD  of $\X_n$ converges weakly to the semi-circle law.

To see \eqref{Markov4}
we compute:
\begin{equation}
A_4 := \frac{1}{n^{2k}} \sum_{\pi^{(1)},\ldots,\pi^{(4)} \in \mathcal{P}(k)} \sum_{P^{(i)}\in S_n\left(\pi^{(i)}\right), i=1,\ldots,4} \E\Big[\prod_{j=1}^{4} \left(a (P^{(j)}) - \E\left[a (P^{(j)})\right]\right)\Big].
\label{eq1}
\end{equation}
Here we write
$P = (P_1,\ldots,P_k) = ( (p_1,q_1), \ldots, (p_k,q_k) )$ and $ a (P) = a(P_{1})\cdot \ldots \cdot a(P_{k})$.
Now consider a partition $\boldsymbol{\pi}$ of $\left\{1,\ldots,4k\right\}$. We call $(P^{(1)},\ldots,P^{(4)})$ $\bpi$ consistent if each $P^{(i)}, i=1,\ldots,4$, is a consistent sequence and
$\big|q_{l}^{(i)} - p_{l}^{(i)}\big| \ = \ \left|q_{m}^{(j)} - p_{m}^{(j)}\right| \Longleftrightarrow l + (i-1) k \ \sim_{\boldsymbol{\pi}} \ m + (j-1) k.
$
Let $\mathcal{S}_n (\bpi)$ denote the set of all $\bpi$ consistent sequences with entries in $\left\{1,\ldots,n\right\}$. Then, \eqref{eq1} becomes
\begin{equation}
A_4 = \frac{1}{n^{2k}} \sum_{\bpi\in\mathcal{P}(4k)} \sum_{(P^{(1)}, \ldots, P^{(4)})\in \mathcal{S}_n\left(\bpi\right)} \E\Big[\prod_{j=1}^{4} \left(a (P^{(j)}) - \E\left[a (P^{(j)})\right]\right)\Big].
\label{eq2}
\end{equation}
Fix a $\bpi \in \mathcal{P}(4k)$. We call $\boldsymbol{\pi}$ a matched partition if
\begin{enumerate}
 \item any equivalence class of $\bpi$ contains at least two elements and
 \item for any $i\in\left\{1,\ldots,4\right\}$ there is a $j\neq i$ and $l,m\in\left\{1,\ldots,k\right\}$ with
	 $l + (i-1) k \ \sim_{\boldsymbol{\pi}} \ m + (j-1) k.$
\end{enumerate}
If $\boldsymbol{\pi}$ is not matched,
$\sum_{(P^{(1)}, \ldots, P^{(4)})\in \mathcal{S}_n\left(\bpi\right)} \E\Big[\prod_{j=1}^{4} \left(a (P^{(j)}) - \E\left[a (P^{(j)})\right]\right)\Big] = 0.$

\noindent 
Therefore, let $\boldsymbol{\pi}$ be a matched partition and denote by $r = \#\bpi$ the number of equivalence classes of $\bpi$. Condition $(i)$ gives $r\leq 2k$. To count all $\bpi$ consistent sequences $(P^{(1)},\ldots,P^{(4)})$, first choose one of at most $n^r$ possibilities to fix the $r$ different equivalence classes. Afterwards, we fix the elements $p_1^{(1)},\ldots,p_1^{(4)}$, which can be done in $n^4$ ways. Since now the differences $|q_{l}^{(i)} - p_{l}^{(i)}|$ are uniquely determined by the choice of the corresponding equivalence classes, we can go on to see that there are at most two choices left for any pair $P_l^{(i)}$. Altogether there are
$ 2^{4k} \cdot n^4 \cdot n^r = C \cdot n^{r+4}$
possibilities to choose $(P^{(1)},\ldots,P^{(4)})$. If now $r\leq 2k-2$, we can conclude that
\begin{equation}
 \# \mathcal{S}_n(\bpi) \leq C \cdot n^{2k+2}.
\label{eq3}
\end{equation}
Hence, it remains to consider the case where $r=2k-1$ and $r=2k$, respectively.

 For $r=2k-1$, we have either two equivalence classes with three elements or one equivalence class with four. For $\bpi$ that is matched, there must be $i\in\left\{1,\ldots,4\right\}$ and $l\in\left\{1,\ldots,k\right\}$ such that $P_l^{(i)}$ is not equivalent to any other pair in the sequence $P^{(i)}$. W.l.o.g., we assume that $i=1$.
Now our approach slightly differs from the one above: 
We now fix all equivalence classes except the one of $P_l^{(1)}$. There are $n^{r-1}$ possibilities to accomplish that. Now we choose again one of $n^4$ possible values for $p_1^{(1)},\ldots,p_1^{(4)}$. Then we fix $q_m^{(1)}$, $m=1,\ldots,l-1$, and start from $q_k^{(1)} = p_1^{(1)}$ to go backwards and obtain the values of $p_{k}^{(1)}, \ldots, p_{l+1}^{(1)}$. Each of these steps admits at most two choices, i.e. $2^{k-1}$ possibilities in total. But now, $P_l^{(1)}$ is uniquely determined since $p_l^{(1)} = q_{l-1}^{(1)}$ and $q_l^{(1)} = p_{l+1}^{(1)}$ by consistency. Thus, we had to make one choice less than before, implying \eqref{eq3}.

Finally, consider the case $r=2k$. Then each equivalence class has exactly two elements. Again we can find $l\in\left\{1,\ldots,k\right\}$ such that $P_l^{(1)}$ is not equivalent to any other pair in the sequence $P^{(1)}$. But additionally, we also have $m\in\left\{1,\ldots,k\right\}$ such that, possibly after relabeling, $P_m^{(2)}$ is neither equivalent to any element in $P^{(1)}$ nor to any other element in $P^{(2)}$. Thus, we can use the same argument as before to see that this time, we can reduce the number of choices to at most $C \cdot n^{r+2} = C \cdot n^{2k+2}$. In conclusion, \eqref{eq3} holds for any matched partition $\bpi$. Thus we obtain \eqref{Markov4}.\\

Finally, let us see that for diagonals with correlations of the same sign that are bounded away from 0 the ESD does not converge to the semi-circle law. To this end it suffices to consider the fourth moment. By the considerations (and with the notation) in the first part of the proof:
\begin{equation*}
 \lim_{n\to\infty} \frac{1}{n} \E\left[\tr\left(\X_{n}^{4}\right)\right] = \lim_{n\to\infty} \frac{1}{n^{3}} \sum_{\pi\in\mathcal{P}\mathcal{P}(4)} \sum_{\left(P_1,\ldots,P_4\right)\in S_{n}^{*}(\pi)} \E\left[a(P_1)\cdot \ldots \cdot a(P_4)\right],
\end{equation*}
and we have seen in Lemmas \ref{le1} and \ref{le2} that the non-crossing partitions contribute exactly $C_2$, i.e. the fourth moment of the semi-circle law.
We will therefore show that the crossing-partitions have non-vanishing contribution and we are done.

\begin{lemma}
If the diagonals have correlations of the same sign that are bounded away from 0, for any crossing $\pi \in \mathcal{P}\mathcal{P}(4) \backslash \mathcal{NPP}(4)$, we have
\begin{equation*}
 \sum_{\left(P_1,\ldots,P_4\right)\in S_{n}^{*}(\pi)} \E\left[a(P_1)\cdot \ldots \cdot a(P_4)\right] \ge K n^{3}.
\end{equation*}
for some constant $0<K< \frac 23$.
\label{crossing2}
\end{lemma}

\begin{proof}
The estimates resemble those in the proof of Lemma \ref{crossing}. However, this time we need a lower bound.
We will assume that all correlations are positive (the all negative part is done identically).
Let $\pi$ be crossing. As ist shown in \cite{hammond_miller}, Lemma 2.5 (also cf. \cite{BCG03}, Section 4.3) the number of $\pi$-consistent sequences $\left(P_1,\ldots,P_4\right)\in S_{n}^{*}(\pi)$ is bounded from below by $\frac 2 3 N^3$. Following the lines of the proof of Lemma \ref{crossing} and using the assumptions on the covariances yields the result.
\end{proof}
\begin{example}
\item Assume that the random variables $\left\{a(p,q), 1\leq p\leq q< \infty\right\}$
satisfy $\mathbb{E}\left[a(p,q)\right]=0$, $\mathbb{E}\left[a(p,q)^{2}\right] = 1$ as well as condition \eqref{moment}. Moreover for all $r$ let $Z^{(r)}_p:=(a(p,p+r))_{p \in \Z}$ be a stationary, {\it (strongly) mixing} process and assume that the processes on the upper diagonals are independent for distinct values of $r$. Then the $\left\{a(p,q), 1\leq p\leq q< \infty\right\}$ satisfy the assumptions of Theorem \ref{main}, since strong mixing implies ergodicity as well as decay of correlations to 0.
\end{example}

\end{document}